\documentclass[12pt,reqno]{amsart}
\usepackage[cp1251]{inputenc}
\usepackage[T2A]{fontenc}
\usepackage[english]{babel}
\usepackage{geometry}

\usepackage{amsmath, amsthm, amscd, amsfonts, amssymb, graphicx, color}
\usepackage[bookmarksnumbered, colorlinks, plainpages]{hyperref}

\numberwithin{equation}{section}

\newtheorem{theorem}{Theorem}[section]

\newtheorem{lemma}[theorem]{Lemma}
\newtheorem{proposition}[theorem]{Proposition}

\theoremstyle{remark}
\newtheorem{remark}[theorem]{Remark}

\theoremstyle{definition}

\newtheorem*{main-definition}{Main Definition}
\newtheorem{condition}[theorem]{Condition}

\begin{document}

\title[Matrix parabolic problems]{Matrix parabolic problems in\\ Sobolev spaces of generalized smoothness}


\author[V. Los]{Valerii Los}

\address{National Technical University of Ukraine "Igor Sikorsky Kyiv Polytechnic Institute", Prospect Peremohy 37, 03056, Kyiv, Ukraine}

\email{v\_los@yahoo.com}


\author[V. Mikhailets]{Vladimir  Mikhailets}

\address{King's College London, Strand, London, WC2R 2LS, United Kingdom; Institute of Mathematics, National Academy of Sciences of Ukraine, 3 Tereshchenkivs'ka, Kyiv, 01024, Ukraine}

\email{mikhailets@imath.kiev.ua}


\author[A. Murach]{Aleksandr Murach}

\address{Institute of Mathematics of the National Academy of Sciences of Ukraine, 3 Tereshchen\-kivs'ka, Kyiv, 01024, Ukraine}

\email{murach@imath.kiev.ua}



\begin{abstract}
We study a general linear parabolic problem for Petrovskii parabolic differential system in Sobolev anisotropic distribution spaces of generalized smoothness. Slowly varying functions are used to characterize supplementary generalized smoothness that cannot be determined by number indexes. We prove that this problem induces topological isomorphisms on appropriate pairs of such spaces. As an application, we give sufficient and necessary conditions for the problem solutions to have prescribed generalized regularity expressed in terms of these spaces. Their use allows obtaining exact conditions for indicated generalized partial derivatives of the solutions to be continuous.
\end{abstract}

\maketitle

\section{Introduction}\label{sec-intr}

Parabolic differential equations, along with hyperbolic ones, form one of the two most important classes of evolution partial differential equations (PDEs). They describe diffusion processes and are closely related to semigroup theory and theory of stochastic processes \cite{Amann95, Barbu22, Casteren11, LorenziRhandi21, Lunardi95}. Therefore, the study of conditions for existence, uniqueness, and regularity of solutions to various parabolic equations is of great interest for applications. As is known \cite{AgranovichVishik64, Eidelman94, EidelmanZhitarashu98, LadyzhenskajaSolonnikovUraltzeva67, LionsMagenes72ii, Solonnikov65}, parabolic initial-boundary-value problems are well posed (in the sense of Hadamard) on appropriate pairs of certain anisotropic Sobolev spaces and H\"older spaces. This property plays a key role in the study of regularity of solutions to parabolic problems, their Green functions, controllability and inverse problems for parabolic systems, parabolic problems with dynamic boundary conditions or stochastic data, some nonlinear parabolic problems, and others (see., e.g., \cite{DenkPrussZacher08, Eidelman94, Ivasyshen90, LadyzhenskajaSolonnikovUraltzeva67, LeBrisLions19, LionsMagenes72ii, Lunardi95, SchnaubeltVeraar11}. In recent years, new classes of function/distribution spaces have been involved in the theory of parabolic problems, which is stimulated by their various applications. Such spaces characterize the regularity of distributions more finely (than Sobolev and H\"older spaces) by means of supplementary smoothness (or differentiability) indexes, integrability exponents, weighted functions, and their combinations \cite{DenkHieberPruss07, DongKim15, Hummel21JEE, HummelLindemulder22, Lindemulder20, LindemulderVeraar20, LosMikhailetsMurach17CPAA, LosMikhailetsMurach21CPAA, Weidemaier05}. For instance, the use of certain functions as a smoothness index for Sobolev spaces leads to more subtle sufficient conditions for generalized solutions of parabolic problems to have prescribed classical differentiability \cite{DiachenkoLos23UMJ8, DiachenkoLos25JMathSci, Los16UMJ11, LosMikhailetsMurach21CPAA}.

This article studies general linear matrix parabolic problems in Sobolev-type distribution spaces parameterized by three smoothness indexes. The first two indexes are real numbers that characterize the different smoothness of distributions in spatial variables and in time (as in the case of classical anisotropic Sobolev spaces). The third index is a slowly varying function and characterizes supplementary generalized smoothness that cannot be determined by number indexes. For example, the logarithmic function, its arbitrary powers and iterations, and their products are admissible as this index. To focus on new effects caused by generalized smoothness, we restrict the study to Hilbert spaces based on the Lebesgue space of square integrable functions. This is also stipulated by our key investigation method, the quadratic interpolation (with function parameter) between  Hilbert spaces.

This article finalizes a series of our papers devoted to properties of various classes of parabolic problems in Sobolev spaces of generalized smoothness \cite{DiachenkoLos22JEPE, Los15UMJ5, Los16UMJ6, Los16UMJ9, Los17UMJ3, LosMurach13MFAT2, LosMurach17OpenMath, LosMikhailetsMurach17CPAA, LosMikhailetsMurach21CPAA, LosMikhailetsMurach23Monograph}. We will prove that general matrix parabolic problems are well posed on appropriate pairs of such spaces and give applications of this result to the study of conditions for generalized or classical regularity of solutions to these problems. It is remarkable that the use of a slowly varying function as the supplementary smoothness index for Sobolev spaces allows us to achieve the minimal values of the number smoothness indexes in conditions for classical differentiability.

Note that the first spaces of generalized smoothness (characterized by a function parameter) were introduced by Malgrange \cite{Malgrange57}
and H\"ormander \cite[Section~2.2]{Hermander63} for the purpose of application to partial differential equations. Sobolev spaces of generalized smoothness have found various applications to elliptic differential operators and elliptic boundary-value problems (see monograph \cite{MikhailetsMurach14}, surveys \cite{MikhailetsMurach12BJMA2, MilkhailetsMurachChepurukhina23UMJ}, and recent articles \cite{AnopDenkMurach21CPAA, AnopChepurukhinaMurach21Axioms, AnopMurach26BJMA, Faierman20, MikhailetsMurach24ProcEdin, MikhailetsMurachZinchenko25AdvMath, Milatovic24, MurachZinchenko25CMP}), oblique derivative problem \cite{Paneah00}, nonlocal boundary-value problems \cite{IlkivStrap15, IlkivStrapVolyanska20, IlkivStrapVolyanska23}, integral equations \cite{MazyaShaposhnikova09, Miyazaki91}, stochastic processes \cite{Jacob010205}, theory of approximation \cite{Stepanets05}. As for applications to some parabolic integro-differential equations, we refer to \cite{MikuleviciusPhonsom19, MikuleviciusPhonsom21}.

The article consists of 7 sections. Section~\ref{sec-intr} is Introduction. The statement of the  matrix parabolic problem under study is given in Section~\ref{sec-prob-state}. We investigate a general linear parabolic problem for Petrovskii parabolic differential system given in a bounded  multidimensional cylinder of smooth cross-section. Such systems are featured by the property that the orders of time derivatives are multiplied by an even number $2b$ when calculating the order $2b\varkappa_{k}$ of partial differential operators acting on $k$-th unknown scalar function. Section~\ref{sec-spaces} is devoted to anisotropic Sobolev spaces of generalized smoothness formed by solutions and right-hand sides of the parabolic system and to their versions containing the boundary and initial data of the problem. These spaces are based on those introduced by H\"ormander \cite[Section~2.2]{Hermander63}. Section~\ref{sec-operator} discusses compatibility conditions for the right-hand sides of the parabolic problem. These conditions are necessary to describe the range of the operator induced by the problem. Main results are formulated in Section~\ref{sec-results}. The central result, Theorem~\ref{25th4.1}, states that the parabolic problem under study induces topological isomorphisms on appropriate pairs of the above-mentioned Sobolev spaces of generalized smoothness, i.e. the problem is well posed on these pairs. As applications of this result, Theorems \ref{25th4.2} and \ref{25th4.3} give sufficient conditions for the problem solutions to have prescribed global or local regularity expressed in terms of Sobolev spaces of generalized smoothness, these conditions being also necessary. The use of such spaces allows obtaining exact conditions for indicated generalized partial derivatives of these solutions to be continuous on a given set. These conditions are formulated as Theorem~\ref{25th4.4}. Section~\ref{sec-auxiliary-results} discusses the method of quadratic interpolation with function parameter between Hilbert spaces and relevant interpolation properties of Sobolev and abstract Hilbert spaces. This method plays a key role in our proof of the central result. We also present an anisotropic version of H\"ormander embedding theorem for spaces of generalized smoothness. We need it to prove Theorem~\ref{25th4.4}. The proofs are given in Section~\ref{sec-proofs}.

\section{The problem statement}\label{sec-prob-state}

Let $2\leq n\in\mathbb{Z}$ and $0<\tau\in\mathbb{R}$. Assume that $G$ is a bounded open domain in $\mathbb{R}^{n}$ with infinitely smooth boundary $\Gamma:=\partial G$. Put $\Omega:=G\times(0,\tau)$ and $S:=\Gamma\times(0,\tau)$; thus, $\Omega$ is an open cylinder in $\mathbb{R}^{n+1}$, and $S$ is its lateral surface.
As usual, $G$ is identified with the lower base of~$\Omega$.

In $\Omega$, we consider a general Petrovskii parabolic problem consisting of $N$ PDEs
\begin{equation}\label{4f1}
\begin{gathered}
\sum_{k=1}^{N}A_{j,k}(x,t,D_x,\partial_t)u_{k}(x,t)=f_{j}(x,t)\\
\mbox{whenever $x\in\Omega$, $0<t<\tau$, and $j\in\{1,\dots,N\}$},
\end{gathered}
\end{equation}
boundary conditions
\begin{equation}\label{4f2}
\begin{gathered}
\sum_{k=1}^{N}B_{j,k}(x,t,D_x,\partial_t)u_{k}(x,t)=g_{j}(x,t)\\
\mbox{whenever $x\in\Gamma$, $0<t<\tau$, and $j\in\{1,\dots,m\}$},
\end{gathered}
\end{equation}
and initial conditions
\begin{equation}\label{4f3}
\begin{gathered}
\partial^{\,r}_t u_{k}(x,t)
\big|_{t=0}=h_{k,r}(x)\quad\mbox{whenever $x\in G$,}\\
\mbox{$k\in\{1,\ldots,N\}$ and $r\in\{0,\ldots,\varkappa_k-1\}$}.
\end{gathered}
\end{equation}
Here,
\begin{equation}\label{4f4}
A_{j,k}:=A_{j,k}(x,t,D_x,\partial_t):=
\sum_{|\alpha|+2b\beta\leq 2b\varkappa_{k}}
a^{\alpha,\beta}_{j,k}(x,t)\,D^\alpha_x\partial^\beta_t
\end{equation}
and
\begin{equation}\label{4f5}
B_{j,k}:=B_{j,k}(x,t,D_x,\partial_t):=
\sum\limits_{|\alpha|+2b\beta\leq l_j+2b\varkappa_{k}}
b^{\alpha,\beta}_{j,k}(x,t)\,D^\alpha_x\partial^\beta_t
\end{equation}
for all admissible values of $j$ and $k$, with $B_{j,k}=0$ if $l_j+2b\varkappa_{k}<0$. Stating this problem, we arbitrarily choose natural numbers $N\geq2$, $b$, and $\varkappa_{1},\ldots,\varkappa_{N}$, put $m:=b(\varkappa_{1}+\dots+\varkappa_N)$, and then arbitrarily choose $m$ integers $l_1,\ldots,l_m$. The number $2b$ is said to be the parabolic weight of the problem. All coefficients of the linear partial differential operators (PDOs) $A_{j,k}$ and $B_{j,k}$ satisfy $a_{j,k}^{\alpha,\beta}\in
C^{\infty}(\overline{\Omega})$ and $b_{j,k}^{\alpha,\beta}\in C^{\infty}(\overline{S})$, where $\overline{\Omega}=\overline{G}\times[0,\tau]$ and
$\overline{S}=\Gamma\times[0,\tau]$. We consider complex-valued functions and distributions and complex linear function/distribution spaces.

We use the following notations in \eqref{4f4} and \eqref{4f5}:
$D^\alpha_x:=D^{\alpha_1}_{1}\dots D^{\alpha_n}_{n}$,
$D_{k}:=i\,\partial/\partial{x_k}$, and $\partial_t:=\partial/\partial t$. Here, $x=(x_1,\ldots,x_n)$ is an arbitrary spatial point in $\mathbb{R}^{n}$, and $\alpha=(\alpha_1,...,\alpha_n)$ is a multi-index, with $|\alpha|:=~\alpha_1+...+\alpha_n$. The summation is performed with respect to all integers $\alpha_1,...,\alpha_n,\beta\geq0$ that satisfy the condition written under the summation symbol.

As indicated above, the initial-boundary-value problem \eqref{4f1}--\eqref{4f3} is Petrovskii parabolic in $\Omega$. This means that it satisfies three conditions (see, e.g., \cite[Chapter~1, Section~1]{Solonnikov65}) stated below. To formulate them, we use the principal symbols of PDOs \eqref{4f4} and \eqref{4f5}. These symbols are defined as follows:
\begin{equation*}
A^{\circ}_{j,k}(x,t,\xi,p):=\sum_{|\alpha|+2b\beta=2b\varkappa_{k}}
a^{\alpha,\beta}_{j,k}(x,t)\,\xi^\alpha p^\beta,
\end{equation*}
and
\begin{equation*}
B^{\circ}_{j,k}(x,t,\xi,p):=
\sum\limits_{|\alpha|+2b\beta=l_j+2b\varkappa_{k}}
b^{\alpha,\beta}_{j,k}(x,t)\,\xi^\alpha p^\beta,
\end{equation*}
with $B_{j,k}(x,t,\xi,p)\equiv0$ if $l_j+2b\varkappa_{k}<0$. They are homogeneous polynomials with respect to the set of the arguments
$\xi:=(\xi_{1},\ldots,\xi_{n})\in\mathbb{C}^{n}$ and $p\in\mathbb{C}$; as usual, $\xi^\alpha:=\xi_{1}^{\alpha_{1}}\ldots\xi_{n}^{\alpha_{n}}$. Define the matrices
\begin{equation*}
A^{\circ}(x,t,\xi,p):=\bigl(A^{\circ}_{j,k}(x,t,\xi,p)\bigr)_{j,k=1}^{N}
\end{equation*}
and
\begin{equation*}
B^{\circ}(x,t,\xi,p):=\bigl(B^{\circ}_{j,k}(x,t,\xi,p)\bigr)
_{\substack{j=1,\ldots,m\\k=1,\ldots,N}}.
\end{equation*}

\begin{condition}\label{9cond2.1sys}
For any $x\in\overline{G}$, $t\in[0,\tau]$ and $\xi\in\mathbb{R}^{n}$, all roots of the polynomial $\det A^{\circ}(x,t,\xi,p)$ in
$p\in\mathbb{C}$ satisfy $\mathrm{Re}\,p(x,t,\xi)\leq -\delta\,|\xi|^{2b}$ with a constant $\delta>0$ not depending on $x$, $t$, and $\xi$.
\end{condition}

\begin{condition}\label{9cond2.2sys}
Each equation \eqref{4f1} with $j\in\{1,\ldots,N\}$ is solvable with respect to the derivative $\partial^{\varkappa_j}_{t}u_{j}$ and does not contain any derivative
$\partial^{\varkappa_k}_{t}u_{k}$ with $k\neq j$. Then we may and do assume that
$a_{j,k}^{(0,\dots,0),\varkappa_{k}}(x,t)\equiv\delta_{j,k}$ is the Kronecker delta for all $j,k\in\{1,\dots,N\}$.
\end{condition}

Let $\delta_1\in(0,\delta)$, with $\delta$ appearing in Condition~\ref{9cond2.1sys}. To formulate the third condition, we arbitrarily choose a point $x\in\Gamma$, number $t\in[0,\tau]$, vector $\xi\in\mathbb{R}^{n}$ tangent to $\Gamma$ at $x$, and number $p\in\mathbb{C}$ subject to $\mathrm{Re}\,p\geq -\delta_1|\xi|^{2b}$ and
$|\xi|+|p|>0$. Let $\nu(x)$ be the unit vector of the inward normal to the boundary $\Gamma$ at $x$. It follows from Condition~\ref{9cond2.1sys} that the polynomial $\det A^{\circ}(x,t,\xi+\zeta\nu(x),p)$ in $\zeta\in\mathbb{C}$ has exactly $m$ roots $\zeta^{+}_{j}(x,t,\xi,p)$, $j=\nobreak1,\ldots,m$, with positive imaginary part and other $m$ roots with negative imaginary part (certainly, with regard for their multiplicity).

\begin{condition}\label{9cond2.3sys}
For a certain number $\delta_1\in(0,\delta)$ and for any $x$, $t$, $\xi$, and $p$ just indicated, the rows of the matrix
\begin{equation*}
B^{\circ}(x,t,\xi+\zeta\nu(x),p)\cdot
\widetilde{A}^{\circ}(x,t,\xi+\zeta\nu(x),p)
\end{equation*}
are linear independent modulo the polynomial $\prod_{j=1}^{m}(\zeta-\zeta^{+}_{j}(x,t,\xi,p))$ in $\zeta\in\mathbb{C}$. Here, $\widetilde{A}^{\circ}(x,t,\xi+\zeta\nu(x),p)$ is the transpose of the cofactor matrix of $A^{\circ}(x,t,\xi+\zeta\nu(x),p)$.
\end{condition}

Conditions \ref{9cond2.1sys} and \ref{9cond2.2sys} means that system \eqref{4f1} is uniformly $2b$-parabolic on the closed cylinder $\overline{\Omega}$ in the sense of Petrovskii \cite[p.~100]{Petrovskii86}. Condition~\ref{9cond2.3sys} means that the collection of boundary conditions \eqref{4f2} covers this system on the lateral surface $\overline{S}$ of the cylinder.

Let us write down \eqref{4f1} and \eqref{4f2} in the matrix form
$Au=f$ in $\Omega$, and  $Bu=g$ on $S$, resp. Here,
$$
A:=(A_{j,k}(x,t,D_x,\partial_t))_{j,k=1}^N
\quad\mbox{and}\quad
B:=\bigl(B_{j,k}(x,t,D_x,\partial_t)\bigr)_{\substack{j=1,\ldots,m
\\k=1,\ldots,N}}
$$
are matrix PDOs, whereas  $u=(u_1,\ldots,u_N)$, $f=(f_1,\ldots,f_N)$, and $g=(g_1,\ldots,g_m)$ are vector-valued functions, we interpreting vectors as columns. The collection of initial conditions \eqref{4f3} is written down in the form $Cu|_{t=0}=h$ on $G$, where
$$
Cu:=(u_{1}(x,t),\partial_t u_{1}(x,t),...,\partial^{\,\varkappa_1-1}_t u_{1}(x,t),...,u_{N}(x,t),\partial_t u_{N}(x,t), ...,\partial^{\,\varkappa_N-1}_t u_{N}(x,t)),
$$
and
$$
h:=(h_{1,0},h_{1,1},...,h_{1,\varkappa_1-1},...,h_{N,0},h_{N,1},...,
h_{N,\varkappa_N-1}).
$$

We associate the linear mapping
\begin{equation}\label{4f9}
u\mapsto\Lambda u:=
\bigl(Au,Bu,Cu\!\mid_{t=0}\bigr),\quad\mbox{where}\;\,u\in
\bigl(C^{\infty}(\overline{\Omega})\bigr)^{N}.
\end{equation}
with the problem \eqref{4f1}--\eqref{4f3}. We study an extension of this mapping on pairs of certain Sobolev spaces of generalized smoothness.

\section{Distribution spaces associated with the problem}\label{sec-spaces}

The above-mentioned spaces are considered in \cite[Section~3]{LosMikhailetsMurach21CPAA} and \cite[Section~1.2]{LosMikhailetsMurach23Monograph}, e.g. They are defined on the base of the anisotropic distribution space  $H^{s,s/(2b);\varphi}(\mathbb{R}^{k})$, which involves three smoothness indexes. The real numbers $s$ and $s/(2b)$ characterize the main smoothness with respect to the first $k-1$ variables (considered as spatial ones) and to the last $k$-th variable (interpreted as time), resp; here, the integer $k\geq2$. Moreover, the function $\varphi$ of a certain class $\mathcal{M}$ characterizes the subordinate generalized smoothness. This is given via the behaviour of the Fourier transform of distributions at infinity. Let us formulate necessary definitions.

By definition, the class $\mathcal{M}$ consists of all Borel measurable functions $\varphi:[1,\infty)\rightarrow(0,\infty)$ that satisfy the following two conditions: both functions $\varphi$ and $1/\varphi$ are bounded on every compact interval $[1,c]$ with $1<c<\infty$, and the function $\varphi$ varies slowly at infinity in the sense of Karamata, i.e.
$$
\lim_{y\rightarrow\infty}\frac{\varphi(\lambda y)}{\varphi(y)}=1 \quad\mbox{for each}\quad\lambda>0.
$$
The theory of slowly varying functions (at infinity) is set forth in monographs  \cite{BinghamGoldieTeugels89, BuldyginIndlekoferKlesovSteinebach18, Seneta76}. The standard  example of such functions is given by the function
\begin{equation*}
\varphi(y):=(\log y)^{\theta_{1}}\,(\log\log y)^{\theta_{2}} \ldots
(\,\underbrace{\log\ldots\log}_{k\;\mbox{\tiny{times}}}y\,)^{\theta_{k}}
\quad\mbox{of}\;\;y\gg1,
\end{equation*}
where the parameters $k\in\mathbb{Z}$, with $k\geq1$, and
$\theta_{1},\theta_{2},\ldots,\theta_{k}\in\mathbb{R}$ are arbitrarily chosen. (Putting here $\varphi(y):=1$ for other $r\geq1$, we get a function of class $\mathcal{M}$.)

Let $s\in\mathbb{R}$ and $\varphi\in\mathcal{M}$. By definition, the linear space $H^{s,s/(2b);\varphi}(\mathbb{R}^{k})$, with $k\geq2$, consists of all tempered distributions $w\in\mathcal{S}'(\mathbb{R}^{k})$ that their (complete) Fourier transform $\widetilde{w}$ is a locally Lebesgue integrable function over $\mathbb{R}^{k}$ and satisfies the condition
\begin{equation}\label{normR}
\begin{split}
&\|w,H^{s,s/(2b);\varphi}(\mathbb{R}^{k})\|:=\\
&\Biggl(\;\int\limits_{\mathbb{R}^{k-1}}\int\limits_{\mathbb{R}}
\bigl(1+|\xi|^2+|\eta|^{1/b}\bigr)^{s}
\,\varphi^{2}\bigl((1+|\xi|^2+|\eta|^{1/b})^{1/2}\bigr)\,
|\widetilde{w}(\xi,\eta)|^{2}\,d\xi\,d\eta\Biggr)^{1/2}<\infty,
\end{split}
\end{equation}
with $\xi\in\mathbb{R}^{k-1}$ and $\eta\in\mathbb{R}$ being frequency variables (dual to spatial variables) and time variable, resp.
This space is endowed with norm \eqref{normR} (and the inner product inducing this norm). It belongs to the classes of distribution spaces introduced and investigated by H\"ormander \cite[Section~2.2]{Hermander63} (see also \cite[Section~10.1]{Hermander83}) and by Volevich and Paneach \cite{VolevichPaneah65}. If $\varphi(\cdot)\equiv1$, then $H^{s,s/(2b);\varphi}(\mathbb{R}^{k})$ becomes anisotropic Sobolev space. In this case, we omit the function parameter $\varphi$ in designations of relevant spaces.

The space $H^{s,s/(2b);\varphi}(\mathbb{R}^{k})$ is Hilbert and separable. The set $C^{\infty}_{0}(\mathbb{R}^{k})$ of compactly supported $C^{\infty}$-functions on $\mathbb{R}^{k}$ is dense in it. Since $t^{\varepsilon}\varphi(t)\to\infty$ and $t^{-\varepsilon}\varphi(t)\to0$ as $t\to\infty$ whatever $\varepsilon>0$ \cite[Section~1.5, Property $1^{\circ}$]{Seneta76}, we have the continuous embeddings
\begin{equation}\label{two-sided-embeddings}
H^{s+\varepsilon,(s+\varepsilon)/(2b)}(\mathbb{R}^{k})\hookrightarrow
H^{s,s/(2b);\varphi}(\mathbb{R}^{k})\hookrightarrow H^{s-\varepsilon,(s-\varepsilon)/(2b)}(\mathbb{R}^{k})
\quad\mbox{for every $\varepsilon>0$};
\end{equation}
i.e. the function parameter $\varphi$ characterizes the subordinate smoothness of distributions with respect to their main anisotropic smoothness defined by the number parameters $s$ and $s/(2b)$. Certainly, the definition of this space makes sense for every real number $b>0$. If $b=1/2$, the space becomes isotropic, is denoted by $H^{s;\varphi}(\mathbb{R}^{k})$, and is also applicable in the $k=1$ case.

If $U$ is a nonempty open subset of $\mathbb{R}^{k}$, then the linear space $H^{s,s/(2b);\varphi}(U)$ consists, by definition, of the restrictions of all distributions $w\in H^{s,s/(2b);\varphi}(\mathbb{R}^{k})$ to $U$ and is endowed with the norm
\begin{equation*}
\|v,H^{s,s/(2b);\varphi}(U)\|:= \inf\bigl\{\|w,H^{s,s/(2b);\varphi}(\mathbb{R}^{k})\|\}
\end{equation*}
of $v\in H^{s,s/(2b);\varphi}(U)$, where the infimum is taken over all $w\in H^{s,s/(2b);\varphi}(\mathbb{R}^{k})$ such that $w=v$ in $U$. The space $H^{s,s/(2b);\varphi}(U)$ is Hilbert and separable (this norm  is induced by an inner product). The set
\begin{equation}\label{c-infty-set}
C^{\infty}_{0}(\overline{U}):=\{w\!\upharpoonright\!\overline{U}:w\in C^{\infty}_{0}(\mathbb{R}^{k})\}
\end{equation}
is dense in this space; as usual, $w\!\upharpoonright\!\overline{U}$ stands for the restriction of $w$ to $\overline{U}$. If $U$ is bounded, then \eqref{c-infty-set} is denoted by $C^{\infty}(\overline{U})$.

Studying the parabolic problem in the cylinder $\Omega$, we need the following examples of $H^{s,s/(2b);\varphi}(U)$: the anisotropic space $H^{s,s/(2b);\varphi}(\Omega)$ (the case where $U=\Omega$ and $k=n+1$) and the isotropic space $H^{s;\varphi}(G)$ (the case where $U=G$, $k=n$, and $b=1/2$). We also need the anisotropic space $H^{s,s/(2b);\varphi}(S)$ defined on the base of $H^{s,s/(2b);\varphi}(\Pi)$ with the help of certain local coordinates on the lateral surface $S$ of $\Omega$; here,  $\Pi:=\mathbb{R}^{n-1}\times(0,\tau)$. Let us give the definition \cite{Los16JMathSci}. It suffices to restrict ourselves to the $s>0$ case.

We arbitrarily choose a finite atlas from $C^{\infty}$-structure on the compact manifold $\Gamma$ (certainly, the structure is induced by $\mathbb{R}^{n}$). Let this atlas be formed by local charts
$\theta_{j}:\mathbb{R}^{n-1}\leftrightarrow\Gamma_{j}$ where $j=1,\ldots,\lambda$. Here, each
$\theta_{j}$ is a $C^{\infty}$-diffeomorphism of whole
$\mathbb{R}^{n-1}$ on an open subset $\Gamma_{j}$ of the manifold $\Gamma$, with $\Gamma:=\Gamma_{1}\cup\cdots\cup\Gamma_{\lambda}$. Moreover, we arbitrarily choose functions
$\chi_{j}\in C^{\infty}(\Gamma)$, with $j=1,\ldots,\lambda$, such that
$\mathrm{supp}\,\chi_{j}\subset\Gamma_{j}$ and $\chi_{1}+\cdots+\chi_{\lambda}=1$ on $\Gamma$.

By definition, the linear space $H^{s,s/(2b);\varphi}(S)$ consists of all functions $\omega\in L_2(S)$ that the function $\omega_{j}(x,t):=\chi_{j}(\theta_{j}(x))\,\omega(\theta_{j}(x),t)$ of $x\in\mathbb{R}^{n-1}$ and $t\in(0,\tau)$ belongs to $H^{s,s/(2b);\varphi}(\Pi)$ for each $j\in\{1,\ldots,\lambda\}$. (Of course, $L_2(S)$ is the Hilbert space of all square-integrable functions over~$S$.) The space $H^{s,s/(2b);\varphi}(S)$ is endowed with the norm
\begin{equation}\label{9f3.8a}
\|\omega,H^{s,s/(2b);\varphi}(S)\|:=
\biggl(\sum_{j=1}^{\lambda}\,
\|\omega_{j},H^{s,s/(2b);\varphi}(\Pi)\|^{2}\biggr)^{1/2}
\end{equation}
and the corresponding inner product. Note that each mapping $\Pi\ni(x,t)\mapsto(\theta_{j}(x),t)$ sets special coordinates on the part $\Gamma_{j}\times(0,\tau)$ of $S$. The space $H^{s,s/(2b);\varphi}(S)$ is Hilbert and separable and does not depend up to equivalence of norms on the choice of $\theta_{j}$ and $\chi_{j}$ \cite[Theorem~1]{Los16JMathSci}. The set $C^{\infty}(\overline{S})$ is dense in this space.

To formulate compatibility conditions for the right-hand parts of the parabolic problem, we also need some spaces over $\Gamma$. The isotropic space $H^{s;\varphi}(\Gamma)$ is defined on the base of $H^{s;\varphi}(\mathbb{R}^{n-1})$ with the help of $\theta_{j}$ and $\chi_{j}$ as follows: $H^{s;\varphi}(\Gamma)$ consists of all distributions $w$ on the closed manifold $\Gamma$ that the distribution $w_{j}:=(\chi_{j}w)\circ\theta_{j}$ belongs to  $H^{s;\varphi}(\mathbb{R}^{n-1})$ for each $j\in\{1,\ldots,\lambda\}$. Here, $(\chi_{j}w)\circ\theta_{j}$ denotes the representation of the distribution $\chi_{j}w$ in the local chart $\theta_{j}$. The space
$H^{s;\varphi}(\Gamma)$ is endowed with the norm
\begin{equation*}
\|w,{H^{s;\varphi}(\Gamma)}\|:=
\biggl(\sum_{j=1}^{\lambda}\,
\|w_{j},H^{s;\varphi}(\mathbb{R}^{n-1})\|^{2}\biggr)^{1/2}
\end{equation*}
and the corresponding inner product. This space is Hilbert and separable and does not depend up to equivalence of norms on the choice of $\theta_{j}$ and $\chi_{j}$ \cite[Theorem~2.3]{MikhailetsMurach14}. The set $C^{\infty}(\Gamma)$ is dense in the space.

Of course, analogues of embeddings \eqref{two-sided-embeddings} hold true for introduced spaces. If these spaces are defined over a bounded set, the corresponding embeddings are compact.

\section{Compatibility conditions for the problem}\label{sec-operator}

They are formulated for the right-hand sides of \eqref{4f1}--\eqref{4f3} and are used to describe the range of the problem operator. These conditions mean that time derivatives of the solutions $u_{k}(x,t)$ at $\nobreak{t=0}$, which can be found from \eqref{4f1} and \eqref{4f3}, should satisfy \eqref{4f2} and some natural relations obtained by the differentiation of \eqref{4f2} with respect to~$t$ (see, e.g., \cite[Chapter~7, Section~10]{LadyzhenskajaSolonnikovUraltzeva67}). The number of such compatibility conditions depends on the solutions smoothness expressed in terms of anisotropic Sobolev spaces. Let us write these conditions.

We put
$$
\sigma_0:=\max\{0,l_1+1,\dots,l_m+1\}
$$
and suppose that $s\in\mathbb{R}$ and $s\geq\sigma_0$. Mapping \eqref{4f9} extends uniquely (by continuity) to a bounded linear operator
\begin{equation}\label{25f4a}
\begin{split}
\Lambda:&\bigoplus_{k=1}^{N}
H^{s+2b\varkappa_k,(s+2b\varkappa_k)/(2b)}(\Omega)\\
\to&\bigl(H^{s,s/(2b)}(\Omega)\bigr)^N\oplus
\bigoplus_{j=1}^{m}H^{s-l_j-1/2,(s-l_j-1/2)/(2b)}(S)
\oplus\bigoplus_{k=1}^{N}\bigoplus_{r=0}^{\varkappa_k-1}
H^{s+2b\varkappa_k-2br-b}(G).
\end{split}
\end{equation}
This follows directly from \cite[Chapter~I, Lemma~4, and Chapter~II, Theorems 3 and 7]{Slobodetskii58}.

We consider any vector-valued function $u(x,t)$ belonging to the domain of this operator and put $(f,g,h):=\Lambda u$ (of course the notation from Section~\ref{sec-prob-state} is used for scalar components of the vector-valued functions $u$, $f$, $g$, and $h$). According to \cite[Chapter~II, Theorem~7]{Slobodetskii58}, there exists a trace
\begin{equation*}
\partial^{\,r}_t u_k(\cdot,0)\in H^{s+2b\varkappa_k-2br-b}(G)
\end{equation*}
for each $k\in\{1,...,N\}$ and every integer $r$ subject to
$$
0\leq r<\frac{s}{2b}+\varkappa_k-\frac{1}{2}
$$
(and only for these values of $r$). This trace is expressed via $f$ and $h$ from system \eqref{4f1} and initial conditions \eqref{4f3}. Indeed, it follows from Condition~\ref{9cond2.2sys} that
\begin{equation*}
\partial^{\varkappa_j}_{t}u_j(x,t)=
-\sum_{k=1}^{N}\sum_{\substack{|\alpha|+2b\beta\leq 2b\varkappa_{k},\\ \beta\leq\varkappa_{k}-1}}
a^{\alpha,\beta}_{j,k}(x,t)\,D^\alpha_x\partial^\beta_t   u_{k}(x,t)+f_{j}(x,t)
\end{equation*}
for all $x\in G$ and $t\in(0,\tau)$ and each $j\in\{1,\dots,N\}$. Differentiating this equality with respect to $t$ and using \eqref{4f3}, we obtain the following recurrent formula:
\begin{equation}\label{29f6}
\begin{split}
(\partial^{r}_t u_j)(x,0)&=h_{j,r}(x)\quad\mbox{for all}\quad j\in\{1,\ldots,N\}\quad\mbox{and}\quad r\in\{0,\dots,\varkappa_j-1\},\\
(\partial^{r}_t u_j)(x,0)&=
-\sum_{k=1}^{N}\!\sum_{\substack{|\alpha|+2b\beta\leq 2b\varkappa_{k}\\ \beta\leq\varkappa_k-1}}
\sum\limits_{q=0}^{r-\varkappa_j}\!
\binom{r-\varkappa_j}{q}(\partial^{r-\varkappa_j-q}_t
a^{\alpha,\beta}_{j,k})(x,0)D^\alpha_x(\partial^{\beta+q}_t
u_k)(x,0)+\\
&+(\partial^{\,r-\varkappa_j}_t f_j)(x,0)\\
&\qquad\mbox{for all}\quad j\in\{1,\dots,N\}\quad\mbox{and}\\
&\qquad r\in\mathbb{Z}\quad\mbox{such that}\quad \varkappa_j\leq r<\frac{s}{2b}+\varkappa_j-\frac{1}{2}.
\end{split}
\end{equation}

Moreover, there exists a trace
$$
\partial^{\,r}_t g_j(\cdot,0)\in H^{s-l_j-1/2-2br-b}(\Gamma)
$$
for each $j\in\{1,\dots,m\}$ and every integer $r$ subject to
$$
0\leq r<\frac{s-l_j-1/2-b}{2b}
$$
(an only for these values of $r$). This trace is expressed in terms of $u$ in the following way from boundary conditions \eqref{4f2}:
\begin{equation}\label{29f4bb}
\begin{aligned}
(\partial^{r}_t g_j)(x,0)&=
\partial^{r}_{t}\biggl(\sum_{i=1}^{N}B_{j,i}u_i(x,t)\biggr)\bigg|_{t=0}\\
&=\sum_{i=1}^{N}\sum_{|\alpha|+2b\beta\leq l_j+2b\varkappa_{i}}\,
\sum_{q=0}^{r}\binom{r}{q}
(\partial^{r-q}_t b^{\alpha,\beta}_{j,i})(x,0)\,
D^\alpha_x(\partial^{\beta+q}_t u_i)(x,0)
\end{aligned}
\end{equation}
for almost all $x\in\Gamma$. Here, each function $(\partial^{\,\beta+q}_{t})u_i(x,0)$ of $x\in G$ is expressed in terms of  $f(x,t)$ and $h(x)$ by means of \eqref{29f6}.

Substituting \eqref{29f6} in the right-hand side of \eqref{29f4bb}, we obtain the compatibility conditions
\begin{equation}\label{25f8}
\begin{gathered}
\partial^{r}_t g_j\!\upharpoonright\!\Gamma=
\mathcal{B}_{j,r}(v_{1,0},\dots,v_{1,[l_j/(2b)]+\varkappa_1+r},\dots,
v_{N,0},\dots,v_{N,[l_j/(2b)]+\varkappa_N+r})\!\upharpoonright\!\Gamma \\
\mbox{for each}\;\;j\in\{1,\dots,m\}\;\;\mbox{and every}\;\;r\in\mathbb{Z}\\
\mbox{such that}\;\;0\leq r<\frac{s-l_j-1/2-b}{2b}.
\end{gathered}
\end{equation}
Here, the functions written between the parentheses are defined as follows:
\begin{equation}\label{29f9}
\begin{split}
v_{i,\mu}(x)&:=h_{i,\mu}(x)\quad\mbox{if}\quad 0\leq \mu\leq\varkappa_i-1,\\
v_{i,\mu}(x)&:=
-\sum_{k=1}^{N}\!\sum_{\substack{|\alpha|+2b\beta\leq 2b\varkappa_{k}\\ \beta\leq\varkappa_k-1}}
\sum\limits_{q=0}^{\mu-\varkappa_i}\!
\binom{\mu-\varkappa_i}{q}(\partial^{\mu-\varkappa_i-q}_t
a^{\alpha,\beta}_{i,k})(x,0)D^\alpha_x v_{k,\beta+q}(x)+\\
&+(\partial^{\,\mu-\varkappa_i}_t f_i)(x,0)\quad\mbox{if}\quad \mu\geq\varkappa_i.
\end{split}
\end{equation}
Moreover, we put
\begin{equation}\label{29f9B}
\begin{aligned}
&\mathcal{B}_{j,r}(v_{1,0},\dots,v_{1,[l_j/(2b)]+\varkappa_1+r},\dots,
v_{N,0},\dots,v_{N,[l_j/(2b)]+\varkappa_N+r})(x):=\\
&:=\sum_{i=1}^{N}\sum_{|\alpha|+2b\beta\leq l_j+2b\varkappa_{i}}\,
\sum_{q=0}^{r}\binom{r}{q}(\partial^{r-q}_t b^{\alpha,\beta}_{j,i})(x,0)\,D^\alpha_{x}v_{i,\beta+q}(x)
\end{aligned}
\end{equation}
for almost all $x\in G$. Of course, if $s\leq l_j+1/2+b$, we have no  compatibility conditions \eqref{25f8} involving $g_j$.

The number of compatibility conditions \eqref{25f8} is a function of  $s\geq\sigma_0$. All points of its discontinuity form the set
\begin{equation}\label{set-E}
E:=\{(2l+1)b+l_j+1/2:j,l\in\mathbb{Z},\,1\leq j\leq m,\,l\geq0\}\cap(\sigma_0,\infty).
\end{equation}

\section{Main results}\label{sec-results}

Let $s\in\mathbb{R}$, $s>\sigma_0$, and $\varphi\in\mathcal{M}$. If $\varphi(\cdot)\equiv1$, we also assume the case where $s=\sigma_0$. Taking \eqref{25f4a} into account, we consider the Hilbert spaces
\begin{equation*}
\bigoplus_{k=1}^{N}
H^{s+2b\varkappa_k,(s+2b\varkappa_k)/(2b);\varphi}(\Omega)
\end{equation*}
and
\begin{align*}
\mathcal{H}^{s,s/(2b);\varphi}
:=\bigl(H^{s,s/(2b);\varphi}(\Omega)\bigr)^N&\oplus
\bigoplus_{j=1}^{m}H^{s-l_j-1/2,(s-l_j-1/2)/(2b);\varphi}(S)\\
&\oplus\bigoplus_{k=1}^{N}\bigoplus_{r=0}^{\varkappa_k-1}\,
H^{s+2b\varkappa_k-2br-b;\varphi}(G).
\end{align*}
The first of them serves as the domain of an isomorphism induced by  parabolic problem \eqref{4f1}--\eqref{4f3}, whereas the second contains the range of this isomorphism. We denote this range by $\mathcal{Q}^{s,s/(2b);\varphi}$ and define it in different ways in the cases where $s\notin E$ and where $s\in E$.

If $s\notin E$, then, by definition, the linear space $\mathcal{Q}^{s,s/(2b);\varphi}$ consists of all vectors $(f,g,h)\in\mathcal{H}^{s,s/(2b);\varphi}$ satisfying compatibility conditions \eqref{25f8} and is endowed with the norm and inner product in $\mathcal{H}^{s,s/(2b);\varphi}$. This space is well defined and complete (i.e. Hilbert) whatever $\varphi\in\mathcal{M}$. This is true for the parameter $\varphi(\cdot)\equiv1$ because the differential and  trace operators used in \eqref{25f8} are bounded on relevant pairs of Sobolev spaces \cite[Chapter~I, Lemma~4, and Chapter~II, Theorems 3 and 7]{Slobodetskii58}. Hence, this is also true for every  $\varphi\in\mathcal{M}$ in view of the continuous embedding
\begin{equation*}
\mathcal{H}^{s-2m,(s-2m)/(2b);\varphi}\hookrightarrow
\mathcal{H}^{s-\varepsilon-2m,(s-\varepsilon-2m)/(2b)}
\end{equation*}
considered for sufficiently small numbers $\varepsilon>0$ such that $s$ and $s-\varepsilon$ belong to the same connected component of $(\sigma_0,\infty)\setminus E$ (the compatibility conditions will be the same for such $s$ and $s-\varepsilon$). Of course,  $\mathcal{Q}^{s,s/(2b);\varphi}=\mathcal{H}^{s,s/(2b);\varphi}$ when  $s\leq\min\{l_1\,\ldots,l_m\}+1/2+b$.

If $s\in E$, then we put
\begin{equation}\label{25f10}
\mathcal{Q}^{s,s/(2b);\varphi}:=
\bigl[\mathcal{Q}^{s-\varepsilon,(s-\varepsilon)/(2b);\varphi},
\mathcal{Q}^{s+\varepsilon,(s+\varepsilon)/(2b);\varphi}\bigr]_{1/2}.
\end{equation}
Here, $[\cdot,\cdot]_{1/2}$ denotes the Hilbert space obtained by the quadratic interpolation with parameter $1/2$ between Hilbert spaces written in brackets, whereas the number $\varepsilon\in(0,1/2)$ is arbitrarily chosen. (We recall the definition of such an interpolation in the next section.) The right-hand side of \eqref{25f10} does not depend up to equivalence of norms on indicated $\varepsilon$ and is continuously embedded in $\mathcal{H}^{s,s/(2b);\varphi}$, which will be justified in Remark~\ref{rem-range-independence}. Of course, $s\mp\varepsilon\notin E$. Hence, the space $\mathcal{Q}^{s,s/(2b);\varphi}$ is well defined in the case where $s\in E$.

\begin{theorem}\label{25th4.1}
Mapping \eqref{4f9} extends uniquely (by continuity) to a topological isomorphism
\begin{equation}\label{25f11}
\Lambda:\bigoplus_{k=1}^{N}
H^{s+2b\varkappa_k,(s+2b\varkappa_k)/(2b);\varphi}(\Omega)\leftrightarrow
\mathcal{Q}^{s,s/(2b);\varphi}
\end{equation}
for each real number $s>\sigma_0$ and every function parameter  $\varphi\in\nobreak\mathcal{M}$.
\end{theorem}

This theorem is known if $\varphi(\cdot)\equiv1$, with $s\geq\sigma_0$. In this case, it is contained in Solonnikov's result \cite[Theorem~5.4]{Solonnikov65} on the assumption that $s/(2b)\in\mathbb{Z}$, and in Zhitarashu's result \cite[Theorem~5.7]{EidelmanZhitarashu98} on the assumption that
$s+1/2\notin\mathbb{Z}$ and $s/(2b)+1/2\notin\mathbb{Z}$.

The following two results present sufficient conditions for a solution of  parabolic problem \eqref{4f1}--\eqref{4f3} to have a prescribed regularity given in terms of Sobolev spaces of generalized smoothness. Let the problem right-hand sides be arbitrary distributions, i.e.
\begin{equation*}
f\in(\mathcal{D}'(\Omega))^{N},\quad g\in(\mathcal{D}'(S))^{m},
\quad\mbox{and}\quad h\in(\mathcal{D}'(G))^{\varkappa_1+\cdots+\varkappa_N}.
\end{equation*}
As usual, $\mathcal{D}'(\cdot)$ denotes the linear topological space of all scalar distributions given on an indicated Euclidean domain or $C^{\infty}$-manifold. A vector-valued function
\begin{equation}\label{25f12a}
u\in\bigoplus_{k=1}^{N}
H^{\sigma_0+2b\varkappa_k,(\sigma_0+2b\varkappa_k)/(2b)}(\Omega)
\end{equation}
is said to be a (strong) generalized solution to problem \eqref{4f1}--\eqref{4f3}  if $\Lambda u=(f,g,h)$, where $\Lambda$ is bounded operator \eqref{25f4a} with $s=\sigma_0$ and $\varphi(\cdot)\equiv1$. In relation to this notion of solution, Isomorphism Theorem~\ref{25th4.1} means that parabolic problem \eqref{4f1}--\eqref{4f3} is well posed (in the sense of Hadamard) on the pair of Hilbert spaces \eqref{25f11} for all $s>\sigma_0$ and $\varphi\in\mathcal{M}$. Specifically, we obtain the following result:

\begin{theorem}\label{25th4.2}
Suppose that vector-valued function \eqref{25f12a} is a generalized solution to parabolic problem \eqref{4f1}--\eqref{4f3} whose right-hand sides satisfy $(f,g,h)\in \mathcal{Q}^{s,s/(2b);\varphi}$
for certain $s>\sigma_0$ and $\varphi\in\mathcal{M}$. Then
$u_k\in H^{s+2b\varkappa_k,(s+2b\varkappa_k)/(2b);\varphi}(\Omega)$ for each $k\in\{1,\ldots,N\}$.
\end{theorem}

Let us formulate a local version of this result. Suppose that $U$ is an open subset of $\mathbb{R}^{n+1}$ such that  $\Omega_0:=U\cap\Omega\neq\emptyset$ and $U\cap\Gamma=\emptyset$.
Put $\Omega':=U\cap\partial\overline{\Omega}$, $S_0:=U\cap S$, $S':=U\cap \{(x,\tau):x\in\Gamma\}$, and $G_0:=U\cap G$. Let $H^{s,s/(2b);\varphi}_{\mathrm{loc}}(\Omega_0,\Omega')$ denote the linear space of all distributions $v\in\mathcal{D}'(\Omega)$ such that $\chi v\in H^{s,s/(2b);\varphi}(\Omega)$ for every function $\chi\in C^\infty (\overline\Omega)$ satisfying  $\mathrm{supp}\,\chi\subset\Omega_0\cup\Omega'$. Similarly, let
$H^{s,s/(2b);\varphi}_{\mathrm{loc}}(S_0,S')$ stand for the linear space of all distributions $\omega\in\mathcal{D}'(S)$ such that $\chi\omega\in H^{s,s/(2b);\varphi}(S)$ for each function $\chi\in C^\infty (\overline S)$ subject to $\mathrm{supp}\,\chi\subset S_0\cup S'$. Finally, let $H^{s;\varphi}_{\mathrm{loc}}(G_0)$ denote the linear space of all distributions $w\in\mathcal{D}'(G)$ such that $\chi w\in H^{s;\varphi}(G)$ for every function $\chi\in C^\infty (\overline G)$ satisfying $\mathrm{supp}\,\chi\subset G_0$. Of course, the $\mathrm{loc}$-spaces are defined for the same $s$ as their base counterparts.

\begin{theorem}\label{25th4.3}
Let $s>\sigma_0$ and $\varphi\in\mathcal{M}$. Suppose that vector-valued function \eqref{25f12a} is a generalized solution to parabolic problem \eqref{4f1}--\eqref{4f3} whose right-hand sides satisfy the following conditions:
\begin{gather}\label{25f13}
f_j\in H^{s,s/(2b);\varphi}_{\mathrm{loc}}(\Omega_0,\Omega')
\quad\mbox{for each}\;\;j\in\{1,\dots,N\},\\
g_{j}\in H^{s-l_j-1/2,(s-l_j-1/2)/(2b);\varphi}_{\mathrm{loc}}(S_0,S')
\quad\mbox{for each}\;\;j\in\{1,\dots,m\},\label{25f14}\\
h_{k,r}\in H^{s+2b\varkappa_k-2br-b;\varphi}_{\mathrm{loc}}(G_0)
\quad\mbox{for all}\;\;k\in\{1,\dots,N\}\;\;\mbox{and}\;\; r\in\{0,\dots,\varkappa_k-1\}.
\label{25f15}
\end{gather}
Then $u_k\in H^{s+2b\varkappa_k,(s+2b\varkappa_k)/(2b);\varphi}_{\mathrm{loc}}
(\Omega_0,\Omega')$ for each $k\in\{1,\ldots,N\}$.
\end{theorem}

Note that the assumption $U\cap\Gamma=\emptyset$ allows us to omit  the compatibility conditions in the hypotheses of this theorem.

The hypotheses of Theorems \ref{25th4.2} and \ref{25th4.3} are also necessary. This follows directly from Theorem~\ref{25th4.1}, with the assumption $U\cap\Gamma=\emptyset$ being not essential for the necessity.

The following two theorems give new sufficient conditions for scalar components of the generalized solution $u$ to be continuous on the closed cylinder $\overline{\Omega}$ or its part $\Omega_0\cup\Omega'$ and to have prescribed continuous partial derivatives there.

\begin{theorem}\label{25th4.4}
Let $p\in\mathbb{Z}$ and $p+b+n/2>\sigma_0$. Suppose that vector-valued function \eqref{25f12a} is a generalized solution to parabolic problem \eqref{4f1}--\eqref{4f3} whose right-hand sides satisfy the condition $(f,g,h)\in\mathcal{Q}^{s,s/(2b);\varphi}$ (or conditions \eqref{25f13}--\eqref{25f15}) for $s:=p+b+n/2$ and a certain function parameter $\varphi\in\mathcal{M}$ subject to
\begin{equation}\label{9f4.7}
\int\limits_{1}^{\,\infty}\;\frac{dy}{y\,\varphi^2(y)}<\infty.
\end{equation}
Then, whatever $k\in\{1,...,N\}$, the distribution
$D_{x}^{\alpha}\partial_{t}^{\beta}u_k(x,t)$ is continuous on $\overline{\Omega}$ (or $\Omega_0\cup\Omega'$, resp.) provided that $0\leq|\alpha|+2b\beta\leq p+2b\varkappa_k$.
\end{theorem}

\begin{remark}\label{rem-continuous-interpretation}
Of course, the function $u_k(x,t)$ (represented by the case where $\alpha=0$ and $\beta=0$) and its generalized partial derivatives are understood in the sense of the theory of distributions (given in $\Omega$). As for the conclusion of Theorem~\ref{25th4.4}, recall that a distribution $v\in\mathcal{D}'(\Omega)$ is called continuous on $\Omega_0\cup\Omega'$ if there exists a continuous function $v_{0}$ on $\Omega_0\cup\Omega'$ such that
\begin{equation*}
v(\omega)=\int\limits_{\Omega_0}v_{0}(x,t)\,\omega(x,t)\,dxdt
\end{equation*}
for every function $\omega\in C^{\infty}(\Omega)$ subject to $\mathrm{supp}\,\omega\subset\Omega_0$, with $v(\omega)$ denoting the value of the functional $v$ at $\omega$. This definition is equivalent to the following: $\chi v\in C(\overline\Omega)$ for every function $\chi\in C^\infty (\overline\Omega)$ such that $\mathrm{supp}\,\chi\subset\Omega_0\cup\Omega'$.
\end{remark}

\begin{remark}\label{rem-sharp-condition}\rm
Condition \eqref{9f4.7} is exact in Theorem~\ref{25th4.4} if (of course) $p+2b\max\{\varkappa_1,\ldots,\varkappa_N\}\geq0$. Namely, let $s:=p+b+n/2$ and $\varphi\in\mathcal{M}$, and suppose that, for every vector-valued function \eqref{25f12a}, the following implication holds true:
\begin{equation}\label{sharp-condition-implication-global}
\begin{aligned}
&\bigl(u\;\mbox{is a solution to problem \eqref{4f1}--\eqref{4f3} for certain $(f,g,h)\in\mathcal{Q}^{s,s/(2b);\varphi}$}\bigr)\\
&\Longrightarrow\bigl(u\;\mbox{satisfies the conclusion of Theorem~\ref{25th4.4}}\bigr).
\end{aligned}
\end{equation}
Then $\varphi$ satisfies condition \eqref{9f4.7}. This inference is also true if we replace \eqref{sharp-condition-implication-global} with the implication
\begin{equation}\label{sharp-condition-implication-local}
\begin{aligned}
&\bigl(u\;\mbox{is a solution to problem \eqref{4f1}--\eqref{4f3} for some right-hand sides \eqref{25f13}--\eqref{25f15}}\bigr)\\
&\Longrightarrow\bigl(u\;\mbox{satisfies the conclusion of Theorem~\ref{25th4.4}}\bigr).
\end{aligned}
\end{equation}
\end{remark}

Thus, we cannot achieve the limiting value $s=p+b+n/2$ in the hypotheses of Theorem~\ref{25th4.4} if we use Sobolev spaces only (restricting ourselves to the $\varphi(\cdot)\equiv1$ case).

We will prove Theorems \ref{25th4.1}, \ref{25th4.3}, \ref{25th4.4} and substantiate Remark~\ref{rem-sharp-condition} in Section~\ref{sec-proofs}. In the $N=1$ case of scalar parabolic problems, these theorems were proved in  \cite[Section~6]{LosMikhailetsMurach21CPAA}. Articles \cite{DiachenkoLos22JEPE, DiachenkoLos23UMJ8} considered parabolic problems for systems of second-order PDEs and boundary conditions whose orders do not exceed~$1$. A special case of homogeneous initial conditions was studied in \cite{Los17UMJ3}. In this case, there are no compatibility conditions, and it is necessary to use specific subspaces of the anisotropic spaces considered in Section~\ref{sec-spaces}.

\section{Auxiliary results}\label{sec-auxiliary-results}

This section is devoted to known results which play a key role in our proofs. We will deduce Isomorphism Theorem~\ref{25th4.1} from the  $\varphi(\cdot)\equiv1$ case by the quadratic interpolation with an appropriate function parameter between Hilbert spaces. Let us recall the definition of this interpolation and its properties that we need. This interpolation method first appeared in \cite[p.~278]{FoiasLions61} and is  expounded in monographs \cite{Ameur19, MikhailetsMurach14, Ovchinnikov84, Simon19}. We mainly follow \cite[Section~1.1]{MikhailetsMurach14}. It suffices to restrict ourselves to separable Hilbert spaces.

Let $X:=[X_{0},X_{1}]$ be an ordered pair of separable (complex) Hilbert spaces such that $X_1$ is a dense manifold in $X_0$ and that the embedding $X_{1}\hookrightarrow X_{0}$ is continuous. Such a pair is said to be regular. For $X$ there exists a positive-definite self-adjoint operator $J$ that acts in $X_{0}$, is defined on $X_{1}$, and satisfies $\|Jv\|_{X_{0}} =\|v\|_{X_{1}}$ for every $v\in X_{1}$.

Let $\mathcal{B}$ stand for the set of all Borel measurable functions $\psi:(0,\infty)\rightarrow(0,\infty)$ such that $\psi$ is bounded on the compact interval $[a,b]$ whenever $0<a<b<\infty$ and that $1/\psi$ is bounded on the closed semiaxis $[a,\infty)$ whenever $a>0$. Given a function $\psi\in\mathcal{B}$, we have the self-adjoint (generally, unbounded) operator $\psi(J)$ in the Hilbert space $X_{0}$. This operator (i.e. the Borel function $\psi$ of $J$) is defined with the help of the spectral resolution of $J$. Let $[X_{0},X_{1}]_{\psi}$ or $X_{\psi}$ denote the domain of the operator $\psi(J)$ acting in $X_{0}$. The linear space $X_{\psi}$ is endowed with the inner product $(v_{1},v_{2})_{X_{\psi}}:=(\psi(J)v_{1},\psi(J)v_{2})_{X_{0}}$ of vectors $v_{1},v_{2}\in X_{\psi}$ and the corresponding norm. This space is Hilbert and separable.

A function $\psi\in\mathcal{B}$ is called an interpolation parameter if
it satisfies the following condition for all regular pairs $X=[X_{0},X_{1}]$ and $Y=[Y_{0},Y_{1}]$ of Hilbert spaces and every linear mapping $T$ defined on $X_{0}$: if the restriction of $T$ to $X_{j}$ is a bounded operator $T:X_{j}\rightarrow Y_{j}$ for each $j\in\{0,1\}$, then the restriction of $T$ to $X_{\psi}$ is a bounded operator $T:X_{\psi}\rightarrow Y_{\psi}$. We say in this case that the Hilbert space $X_{\psi}$ is obtained by the quadratic interpolation with the function parameter $\psi$ of the pair $X=[X_{0},X_{1}]$ (or between the spaces $X_{0}$ and $X_{1}$) and that the bounded operator $T:X_{\psi}\rightarrow Y_{\psi}$ is obtained by this interpolation of the operators $T:X_{j}\rightarrow Y_{j}$ with $j\in\{0,1\}$.

A function $\psi\in\mathcal{B}$ is an interpolation parameter if and only if there exists a concave function $\psi_{1}(r)>0$ of $r\gg1$ that the functions $\psi/\psi_{1}$ and $\psi_{1}/\psi$ are bounded in a neighbourhood of infinity. This criterion follows from Peetre's results \cite{Peetre66, Peetre68}.

Specifically, the power function $\psi(r)\equiv r^{\theta}$ is an interpolation parameter if and only if $\nobreak{0\leq\theta\leq1}$. In this case, the exponent $\theta$ serves as a number parameter of the interpolation, and the interpolation space $X_{\psi}$ is also denoted by $X_{\theta}$. This interpolation was used in \eqref{25f10} for $\theta=1/2$. Note that the theory of interpolation between normed spaces begun with the method of the quadratic interpolation with number parameter proposed independently by S.~G.~Krein and J.-L.~Lions.

We need the following result on the quadratic interpolation between Sobolev (isotropic and anisotropic) spaces:

\begin{proposition}\label{8prop5}
Let numbers $s_{0},s,s_{1}\in\mathbb{R}$ satisfy $0\leq s_{0}<s<s_{1}$, and let $\varphi\in\mathcal{M}$. Put
\begin{equation}\label{25f16}
\psi(r):=
\begin{cases}
\;r^{(s-s_{0})/(s_{1}-s_{0})}\,\varphi(r^{1/(s_{1}-s_{0})})&\text{if}
\quad r\geq1,\\
\;\varphi(1) & \text{if}\quad0<r<1.
\end{cases}
\end{equation}
Then the function $\psi\in\mathcal{B}$ is an interpolation parameter, and  the equalities
\begin{equation}\label{25f22}
H^{s+\lambda,(s+\lambda)/(2b);\varphi}(W)=
\bigl[H^{s_{0}+\lambda,(s_{0}+\lambda)/(2b)}(W),
H^{s_{1}+\lambda,(s_{1}+\lambda)/(2b)}(W)\bigr]_{\psi}
\end{equation}
and
\begin{equation}\label{8f49}
H^{s+\lambda;\varphi}(G)=
\bigl[H^{s_{0}+\lambda}(G),H^{s_{1}+\lambda}(G)\bigr]_{\psi}
\end{equation}
hold true with equivalence of norms for any real $\lambda\geq-s_{0}$ provided that $W=\Omega$ or $W=S$ in \eqref{25f22}.
\end{proposition}

Formula \eqref{25f22} is proved in \cite[Theorem~2]{Los16JMathSci} for the case of $W=S$. The $W=\Omega$ case is treated in the same way as \cite[Lemma~2]{Los16JMathSci} relating to spaces on a strip. Formula \eqref{8f49} is contained in \cite[Theorem~3.2]{MikhailetsMurach14}.

Dealing with definition \eqref{25f10}, we will use another interpolation result.

\begin{proposition}\label{16lem7.5}
Let numbers $\sigma,\varepsilon\in\mathbb{R}$ satisfy
$\sigma>\varepsilon>0$, and let $\varphi\in\mathcal{M}$.
Then the equalities
\begin{equation}\label{16f46}
H^{\sigma,\sigma/(2b);\varphi}(W)=
\bigl[H^{\sigma-\varepsilon,(\sigma-\varepsilon)/(2b);\varphi}(W),
H^{\sigma+\varepsilon,(\sigma+\varepsilon)/(2b);\varphi}(W)\bigr]_{1/2}
\end{equation}
and
\begin{equation}\label{interp-half-isotropic-spaces}
H^{\sigma;\varphi}(G)=
\bigl[H^{\sigma-\varepsilon;\varphi}(G),
H^{\sigma+\varepsilon;\varphi}(G)\bigr]_{1/2}
\end{equation}
holds true up with equivalence of norms provided that $W=\Omega$ or $W=S$ in \eqref{16f46}.
\end{proposition}

Formula \eqref{16f46} is proved in \cite[Lemma~6.3]{LosMikhailetsMurach21CPAA}, and formula \eqref{interp-half-isotropic-spaces} is contained in \cite[Theorem~5.2]{MikhailetsMurach15ResMath1}.

We also use two general properties of the interpolation between normed spaces. They are formulated for the quadratic interpolation as follows:

\begin{proposition}\label{8prop1}
Let $X=[X_{0},X_{1}]$ be a regular pair of Hilbert spaces, and let $Y_{0}$ be a subspace of $X_{0}$. Then $Y_{1}:=X_{1}\cap Y_{0}$ is a subspace of $X_{1}$. Assume that there exists a linear mapping $P$ on $X_{0}$ such that $P$ is a projector of $X_{j}$ onto $Y_{j}$ for each $j\in\{0,\,1\}$. Then the pair $[Y_{0},Y_{1}]$ is regular, and $[Y_{0},Y_{1}]_{\psi}=X_{\psi}\cap Y_{0}$ with equivalence of norms for every interpolation parameter~$\psi\in\mathcal{B}$. Here, $X_{\psi}\cap Y_{0}$ is a subspace of $X_{\psi}$.
\end{proposition}

As for the proof of this proposition, we refer to   \cite[Theorem~1.17.1/1]{Triebel95} and \cite[Theorem~1.8]{MikhailetsMurach14}. Certainly, subspaces of Hilbert spaces are supposed to be closed. Generally speaking, we deal with nonorthogonal projectors onto subspaces.

\begin{proposition}\label{8prop2}
Let $[X_{0}^{(j)},X_{1}^{(j)}]$, with $j=1,\ldots,q$, be a finite collection of regular pairs of Hilbert spaces. Then
$$
\biggl[\,\bigoplus_{j=1}^{q}X_{0}^{(j)},\,
\bigoplus_{j=1}^{q}X_{1}^{(j)}\biggr]_{\psi}=\,
\bigoplus_{j=1}^{q}\bigl[X_{0}^{(j)},\,X_{1}^{(j)}\bigr]_{\psi}
$$
with equality of norms for every function $\psi\in\mathcal{B}$.
\end{proposition}

The proof is given in \cite[Theorem~1.8]{MikhailetsMurach14}, e.g.

Our proof of Theorem~\ref{25th4.4} relies on a version of H\"ormander's embedding theorem \cite[Theorem~2.2.7]{Hermander63} for anisotropic spaces considered in Section~\ref{sec-spaces}.

\begin{proposition}\label{prop-embedding}
Let $q\in\mathbb{Z}$, $q\geq0$, $\sigma:=q+b+n/2$, and $\varphi\in\mathcal{M}$. The following two assertions are true:
\begin{itemize}
\item[(i)] If $\varphi$ satisfies \eqref{9f4.7}, then every function $w\in H^{\sigma,\sigma/(2b);\varphi}(\mathbb{R}^{n+1})$ has the following property: this function and all its generalized partial derivatives $D_{x}^{\alpha}\partial_{t}^{\beta}w(x,t)$ with $|\alpha|+2b\beta\leq q$ are continuous on $\mathbb{R}^{n+1}$.
\item[(ii)] Let $V$ be a nonempty open subset of $\mathbb{R}^{n+1}$, and let an integer $i$ be such that $1\leq i\leq n$. If every function $w\in H^{\sigma,\sigma/(2b);\varphi}(\mathbb{R}^{n+1})$ with $\mathrm{supp}\,w\subset V$ satisfies the condition $\partial_{i}^{j}w\in C(\mathbb{R}^{n+1})$ for each $j\in\mathbb{Z}$ with $0\leq j\leq q$, then $\varphi$ satisfies \eqref{9f4.7}. Here, $\partial_{i}^{j}w$ denotes the generalized partial derivative $(\partial^{j}w)/\partial{x_{i}^{j}}$ of the function $w=w(x_{1},\ldots,x_{n},t)$.
\end{itemize}
\end{proposition}

This result is proved in \cite[Lemma~8.1]{LosMikhailetsMurach17CPAA}.

\section{Proofs}\label{sec-proofs}

Our proof of Theorem~\ref{25th4.1} relies on a lemma concerning the quadratic interpolation between the ranges of some isomorphisms \eqref{25f11} considered in the $\varphi(\cdot)\equiv1$ case. Let $\{J_l:1\leq l\in\mathbb{Z}\}$ stand for the collection of all connected components of the set $(\sigma_0,\infty)\setminus E$. Recall that the set $E$ was defined by \eqref{set-E}.

\begin{lemma}\label{25lem4.1}
Let $1\leq l\in\mathbb{Z}$, and suppose that real numbers $s_0,s,s_1\in J_{l}$ satisfy $s_0<s<s_1$ and that $\varphi\in\mathcal{M}$. Let $\psi$ be the interpolation parameter defined by \eqref{25f16}. Then
\begin{equation}\label{25f42}
\mathcal{Q}^{s,s/(2b);\varphi}=\,
\bigl[\mathcal{Q}^{s_0,s_0/(2b)},
\mathcal{Q}^{s_1,s_1/(2b)}\bigr]_{\psi}
\end{equation}
holds true with equivalence of norms.
\end{lemma}

\begin{proof}
Owing to Propositions \ref{8prop5} and \ref{8prop2}, we conclude that
\begin{equation}\label{25f45}
\bigl[\mathcal{H}^{s_0,s_0/(2b)},
\mathcal{H}^{s_1,s_1/(2b)}\bigr]_{\psi}=
\mathcal{H}^{s,s/(2b);\varphi}
\end{equation}
with equivalence of norms. To deduce required relation \eqref{25f42} from \eqref{25f45} with the help of Proposition~\ref{8prop1}, we will build a linear mapping $P$ that acts on the space $\mathcal{H}^{s_0,s_0/(2b)}$ and such that $P$ is a projector of the space $\mathcal{H}^{s_i,s_i/(2b)}$ onto its subspace $\mathcal{Q}^{s_i,s_i/(2b)}$ for each $i\in\{0,1\}$.

Whatever $j\in\{1,\dots,m\}$, the set
\begin{equation*}
\biggl\{r\in\mathbb{Z}:0\leq r<\frac{s-l_j-1/2-b}{2b}\biggr\}
\end{equation*}
does not depend on $s\in J_l$. Recall that this set appeared in compatibility conditions \eqref{25f8}. Let $q_{l,j}^{\star}$ denote the number of its elements in the $s\in J_l$ case, and put $q_{l,j}:=q_{l,j}^{\star}-1$ for convenience. Given $\bigl(f,g,h\bigr)\in\mathcal{H}^{s_0,s_0/(2b)}$, we let
\begin{equation*}
\left\{
  \begin{array}{ll}
    g^{*}_{j}:=g_j & \hbox{if}\quad q_{l,j}=-1,\\
    g^{*}_{j}:=g_j+T(w_{j,0},\dots,w_{j,q_{l,j}})
    & \hbox{if}\quad q_{l,j}\geq0.
  \end{array}
\right.
\end{equation*}
Here,
\begin{align*}
w_{j,0}&:=\mathcal{B}_{j,0}(v_{1,0},\dots,v_{1,[l_j/(2b)]+\varkappa_1},
\dots,v_{N,0},\dots,v_{N,[l_j/(2b)]+\varkappa_N})\!\upharpoonright\!\Gamma
-g_j\!\upharpoonright\!\Gamma,\\
&\dots\\
w_{j,q_{l,j}}&:=\mathcal{B}_{j,q_{l,j}}(v_{1,0},\dots,
v_{1,[l_j/(2b)]+\varkappa_1+q_{l,j}},\dots,v_{N,0},\dots,
v_{N,[l_j/(2b)]+\varkappa_N+q_{l,j}})\!\upharpoonright\!\Gamma
-\partial_{t}^{q_{l,j}}g_j\!\upharpoonright\!\Gamma,
\end{align*}
with the functions $v_{i,\mu}$ and PDOs $\mathcal{B}_{j,r}$ being defined by \eqref{29f9} and \eqref{29f9B}. Moreover, $T$ is a bounded linear operator
\begin{equation*}
T: \bigoplus_{k=0}^{r-1}
H^{\lambda-2bk-b}(\Gamma)\rightarrow H^{\lambda,\lambda/(2b)}(S),
\end{equation*}
with $r:=q_{l,j}^\star$ and $\lambda>2br-b$, such that $T$ is right-inverse to the Cauchy data operator
\begin{equation*}
R:\omega\mapsto\bigl(\omega\!\upharpoonright\!\Gamma,
\partial_{t}\omega\!\upharpoonright\!\Gamma,\dots,
\partial^{r-1}_{t}\omega\!\upharpoonright\!\Gamma\bigr)
\end{equation*}
where $\omega\in H^{\lambda,\lambda/(2b)}(S)$. This operator $T$ exists and does not depend on $\lambda>2br-b$ (see, e.g., \cite[лема~6.1]{LosMikhailetsMurach21CPAA}).

Consider the linear mapping $P:\bigl(f,g,h\bigr)\mapsto\bigl(f,g^*,h\bigr)$ defined for every $(f,g,h\bigr)\in\mathcal{H}^{s_0,s_0/(2b)}$; here, $g^*:=(g^*_1,\dots,g^*_m)$. This mapping is required. Indeed, its restriction to each space $\mathcal{H}^{s_{i},s_{i}/(2b)}$, with $i\in\{0,1\}$, is a bounded operator on this space, which stems from the boundedness of partial differential operators, the trace operator, and the operator $T$ on relevant Sobolev spaces. Moreover, it follows from compatibility conditions \eqref{25f8} that $P(f,g,h)\in\mathcal{Q}^{s_{i},s_{i}/(2b)}$ for every $(f,g,h)\in\mathcal{H}^{s_{i},s_{i}/(2b)}$. Finally, $P(f,g,h)=(f,g,h)$ for any vector $(f,g,h)\in\mathcal{Q}^{s_{i},s_{i}/(2b)}$ because this vector satisfies conditions \eqref{25f8} and since they imply that $w_{j,0}=\cdots=w_{j,q_{l,j}}=0$ and, hence, $g_j^*=g_j$ whatever $j\in\{1,\dots,m\}$.

Owing to Proposition~\ref{8prop1} and formula \eqref{25f45}, we conclude that the following equalities hold true with equivalence of norms:
\begin{align*}
\bigl[\mathcal{Q}^{s_0,s_0/(2b)},
\mathcal{Q}^{s_1,s_1/(2b)}\bigr]_{\psi}&=
\bigl[\mathcal{H}^{s_0,s_0/(2b)},
\mathcal{H}^{s_1,s_1/(2b)}\bigr]_{\psi}\cap
\mathcal{Q}^{s_0,s_0/(2b)}=\\
&=\mathcal{H}^{s,s/(2b);\varphi}\cap
\mathcal{Q}^{s_0,s_0/(2b)}=\mathcal{Q}^{s,s/(2b);\varphi};
\end{align*}
i.e. we have proved required relation \eqref{25f42}. Here, the last equation is true because $s,s_0\in J_l$ and, hence, all elements of  $\mathcal{Q}^{s_0,s_0/(2b)}$ and $\mathcal{Q}^{s,s/(2b);\varphi}$ satisfy  the same compatibility conditions. We also take into account the continuous embedding $\mathcal{H}^{s,s/(2b);\varphi}\hookrightarrow
\mathcal{H}^{s_0,s_0/(2b)}$.
\end{proof}

\begin{proof}[Proof of Theorem~$\ref{25th4.1}$.] Let $s>\sigma_0$ and $\varphi\in\mathcal{M}$. Assume first that $s\notin E$. Then $s\in J_{l}$ for a certain integer $l\geq1$. Choose numbers $s_0,s_1\in J_{l}$ such that $s_0<s<s_1$ and that $s_i+1/2\notin\mathbb{Z}$ and $s_i/(2b)+1/2\notin\mathbb{Z}$ for each $i\in\{0,1\}$. According to Zhitarashu's result  \cite[Theorem~5.7]{EidelmanZhitarashu98}, mapping \eqref{4f9} extends uniquely by continuity to (topological) isomorphisms
\begin{equation*}
\Lambda:\bigoplus_{k=1}^{N}
H^{s_i+2b\varkappa_k,(s_i+2b\varkappa_k)/(2b)}(\Omega)
\leftrightarrow
\mathcal{Q}^{s_i,s_i/(2b)},\quad\mbox{where}\quad i\in\{0,1\}.
\end{equation*}
Applying the quadratic interpolation with function parameter \eqref{25f16} to these isomorphisms, we obtain another isomorphism
\begin{equation}\label{16f50}
\begin{aligned}
\Lambda&:\bigoplus_{k=1}^{N}
H^{s+2b\varkappa_k,(s+2b\varkappa_k)/(2b);\varphi}(\Omega)\\
&\quad=\Biggl[\bigoplus_{k=1}^{N}
H^{s_0+2b\varkappa_k,(s_0+2b\varkappa_k)/(2b)}(\Omega),
\bigoplus_{k=1}^{N}
H^{s_1+2b\varkappa_k,(s_1+2b\varkappa_k)/(2b)}(\Omega)\Biggr]_{\psi}\\
&\leftrightarrow
\bigl[\mathcal{Q}^{s_0,s_0/(2b)},
\mathcal{Q}^{s_1,s_1/(2b)}\bigr]_{\psi}=\mathcal{Q}^{{s,s/(2b)};\varphi}.
\end{aligned}
\end{equation}
Here, the equalities hold true with equivalence of norms due to Propositions \ref{8prop5} and \ref{8prop2} and Lemma~\ref{25lem4.1}. Operator \eqref{16f50} is an extension of mapping \eqref{4f9} by continuity because the set $(C^{\infty}(\overline{\Omega}))^{N}$ is dense in the domain of this operator. The theorem is proved in the case where $s\notin E$.

Suppose now that $s\in E$. We arbitrarily choose a number $\varepsilon\in(0,1/2)$. Since $s\mp\varepsilon\notin E$ and $s-\varepsilon>\sigma_0$, we have proved isomorphisms
\begin{equation*}
\Lambda:\bigoplus_{k=1}^{N}
H^{s\mp\varepsilon+2b\varkappa_k,(s\mp\varepsilon+2b\varkappa_k)/(2b);
\varphi}(\Omega)\leftrightarrow
\mathcal{Q}^{s\mp\varepsilon,(s\mp\varepsilon)/(2b);\varphi}.
\end{equation*}
Applying the quadratic interpolation with number parameter $1/2$ to them, we obtain another isomorphism
\begin{equation}\label{8f38}
\begin{aligned}
\Lambda&:\bigoplus_{k=1}^{N}
H^{s+2b\varkappa_k,(s+2b\varkappa_k)/(2b);\varphi}(\Omega)\\
&\quad=\Biggl[\bigoplus_{k=1}^{N}
H^{s-\varepsilon+2b\varkappa_k,(s-\varepsilon+2b\varkappa_k)/(2b);
\varphi}(\Omega),
\bigoplus_{k=1}^{N}
H^{s+\varepsilon+2b\varkappa_k,(s+\varepsilon+2b\varkappa_k)/(2b);
\varphi}(\Omega)\Biggr]_{1/2}\\
&\leftrightarrow
\bigl[\mathcal{Q}^{s-\varepsilon,(s-\varepsilon)/(2b);\varphi},
\mathcal{Q}^{s+\varepsilon,(s+\varepsilon)/(2b);\varphi}\bigr]_{1/2}=
\mathcal{Q}^{s,s/(2b);\varphi}.
\end{aligned}
\end{equation}
Here, the equalities hold true with equivalence of norms by Propositions \ref{16lem7.5} and \ref{8prop2} and formula \eqref{25f10}.  Operator \eqref{8f38} is an extension of mapping \eqref{4f9} by continuity as explained in the previous case.
\end{proof}

\begin{remark}\label{rem-range-independence}
Let $s\in E$. It follows directly from isomorphism \eqref{8f38} that the space $\mathcal{Q}^{s,s/(2b);\varphi}$ defined by \eqref{25f10} is independent of the number $\varepsilon\in(0,1/2)$. This space  is continuously embedded in $\mathcal{H}^{s,s/(2b);\varphi}$. Indeed, choosing $\varepsilon\in(0,1/2)$, we get the continuous embeddings
\begin{equation*}
\mathcal{Q}^{s\mp\varepsilon-2m,(s\mp\varepsilon-2m)/(2b);\varphi}
\hookrightarrow
\mathcal{H}^{s\mp\varepsilon-2m,(s\mp\varepsilon-2m)/(2b);\varphi}
\end{equation*}
in view of $s\mp\varepsilon\in(0,\sigma_{0})\setminus E$ and the definition of the left-hand spaces. Applying the quadratic interpolation with number parameter $1/2$ to these bounded embedding operators, we obtained the required continuous embedding
\begin{align*}
\mathcal{Q}^{s,s/(2b);\varphi}&=
\bigl[\mathcal{Q}^{s-\varepsilon-2m,(s+\varepsilon-2m)/(2b);\varphi},
\mathcal{Q}^{s+\varepsilon-2m,(s+\varepsilon-2m)/(2b);\varphi}\bigr]_{1/2}\\
&\hookrightarrow
\bigl[\mathcal{H}^{s-\varepsilon-2m,(s\mp\varepsilon-2m)/(2b);\varphi},
\mathcal{H}^{s+\varepsilon-2m,(s\mp\varepsilon-2m)/(2b);\varphi}\bigr]_{1/2}
=\mathcal{H}^{s,s/(2b);\varphi}
\end{align*}
Here, the equalities hold true with equivalence of norms by formula \eqref{25f10} and Propositions \ref{16lem7.5} and \ref{8prop2}.
\end{remark}

\begin{proof}[Proof of Theorem~$\ref{25th4.3}$.]
First we will show that hypotheses \eqref{25f13}--\eqref{25f15} entails the following implication:
\begin{equation}\label{25f54}
\begin{aligned}
&u\in\bigoplus_{k=1}^{N}
H^{s-\lambda+2b\varkappa_k,(s-\lambda+2b\varkappa_k)/(2b);\varphi}_
{\mathrm{loc}}(\Omega_0,\Omega')\\
&\Longrightarrow
u\in\bigoplus_{k=1}^{N}
H^{s-\lambda+1+2b\varkappa_k,(s-\lambda+1+2b\varkappa_k)/(2b);\varphi}_
{\mathrm{loc}}(\Omega_0,\Omega').
\end{aligned}
\end{equation}
Here, $\lambda$ is any positive integer that satisfies   $s-\lambda+1>\sigma_0$.

We arbitrarily choose a function $\chi\in C^\infty(\overline\Omega)$ subject to    $\mbox{supp}\,\chi\subset\Omega_0\cup\Omega'$. For $\chi$, there exists a function $\eta\in C^\infty(\overline\Omega)$ such that $\mbox{supp}\,\eta\subset\Omega_0\cup\Omega'$ and that
$\eta=1$ in a neighbourhood of $\mbox{supp}\,\chi$. Rearranging the differential operators $A$, $B$, and $C$ with the operator of multiplication by $\chi$, we obtain the following:
\begin{equation}\label{16f55}
\begin{aligned}
\Lambda(\chi u)&=\Lambda(\chi\eta u)=
\chi\Lambda(\eta u)+ \Lambda'(\eta u)=\\
&=\chi\Lambda u+\Lambda'(\eta u)=\chi(f,g,h)+
\Lambda'(\eta u).
\end{aligned}
\end{equation}
Here, $\Lambda':=(A',B',C')$, with $A'$, $B'$, and $C'$ being matrix differential operators of the same structure as their unprimed counterparts and consisting of lower order entries. Namely,
\begin{equation*}
A'=(A'_{j,k})_{j,k=1}^N,\quad\quad
B'=\bigl(B'_{j,k}\bigr)_{\substack{j=1,\ldots,m
\\k=1,\ldots,N}},
\end{equation*}
and
\begin{equation*}
C'=(C'_{1,0},...,C'_{1,\varkappa_1-1},...,C'_{N,0},...,C'_{N,\varkappa_N-1}),
\end{equation*}
where
\begin{gather*}
A'_{j,k}(x,t,D_x,\partial_t):=
\sum_{|\alpha|+2b\beta\leq 2b\varkappa_{k}-1}
a^{\alpha,\beta}_{j,k,1}(x,t)\,D^\alpha_x\partial^\beta_t,\\
B'_{j,k}(x,t,D_x,\partial_t):=
\sum\limits_{|\alpha|+2b\beta\leq l_j+2b\varkappa_{k}-1}
b^{\alpha,\beta}_{j,k,1}(x,t)\,D^\alpha_x\partial^\beta_t,
\end{gather*}
with $B'_{j,k}=0$ if $l_j+2b\varkappa_{k}\leq0$, and
\begin{equation*}
C'_{k,0}=0\quad\mbox{and}\quad
C'_{k,r}(x,\partial_t)=\sum_{\beta=0}^{r-1}
c_{k,r,\beta}(x)\partial^{\beta}_t
\end{equation*}
whenever $1\leq k\leq N$ and $1\leq r\leq\varkappa_k-1$. Here, all $a^{\alpha,\beta}_{j,k,1}\in C^{\infty}(\overline{\Omega})$, $b^{\alpha,\beta}_{j,k,1}\in C^{\infty}(\overline{S})$, and $c_{k,r,\beta}\in C^{\infty}(\overline{G})$.

Whatever $\sigma\geq\sigma_0-1$, the linear operator $\Lambda'$ is bounded on the pair of spaces
\begin{equation}\label{16f59}
\Lambda':\bigoplus_{k=1}^{N}H^{\sigma+2b\varkappa_k,(\sigma+2b\varkappa_k)/(2b);\varphi}(\Omega)\to
\mathcal{H}^{\sigma+1,\,(\sigma+1)/(2b);\varphi}.
\end{equation}
If $\varphi(\cdot)\equiv1$, this is due to known properties of differential and trace operators on anisotropic Sobolev spaces \cite[Chapter~I, Lemma~4, and Chapter~II, Theorems 3 and 7]{Slobodetskii58}. The boundedness in the general case of $\varphi\in\mathcal{M}$ follows directly from the $\varphi(\cdot)\equiv1$ case by the quadratic interpolation in view of Propositions~\ref{8prop5} and \ref{8prop2}.

According to hypotheses \eqref{25f13}--\eqref{25f15}, the inclusion $\chi\,(f,g,h)\in\mathcal{H}^{s,s/(2b);\varphi}$ holds true. Owing to \eqref{16f59} with $\sigma:=s-\lambda$, we obtain the implication
\begin{align*}
u\in\bigoplus_{k=1}^{N}
H^{s-\lambda+2b\varkappa_k,(s-\lambda+2b\varkappa_k)/(2b);\varphi}_
{\mathrm{loc}}(\Omega_0,\Omega')\Longrightarrow
\Lambda'(\eta u)\in\mathcal{H}^{s-\lambda+1,(s-\lambda+1)/(2b);\varphi}.
\end{align*}
We hence conclude by \eqref{16f55} that
\begin{equation}\label{16f60}
\begin{aligned}
&u\in\bigoplus_{k=1}^{N}
H^{s-\lambda+2b\varkappa_k,(s-\lambda+2b\varkappa_k)/(2b);\varphi}_
{\mathrm{loc}}(\Omega_0,\Omega')\\
&\Longrightarrow\Lambda(\chi u)\in
\mathcal{H}^{s-\lambda+1,(s-\lambda+1)/(2b);\varphi}.
\end{aligned}
\end{equation}

Let us show that
\begin{equation}\label{16f61}
\Lambda(\chi u)\in\mathcal{H}^{\sigma,\sigma/(2b);\varphi}
\;\Longrightarrow\;
\Lambda(\chi u)\in\mathcal{Q}^{\sigma,\sigma/(2b);\varphi}
\end{equation}
whatever $\sigma>\sigma_0$. This implication allows us to prove \eqref{25f54} in view of Theorem~\ref{25th4.2}.

Suppose that $\Lambda(\chi u)\in\mathcal{H}^{\sigma,\sigma/(2b);\varphi}$ for certain $\sigma>\sigma_0$. Since $\mathrm{dist}(\mathrm{supp}\,\chi,\Gamma)>0$, the equality $\Lambda(\chi u)=0$ is valid in a certain neighbourhood of~$\Gamma$. Hence, the vector $\Lambda(\chi u)$ satisfies compatibility conditions \eqref{25f8} in which $(f,g,h):=\Lambda(\chi u)$ and $s:=\sigma$. Thus, $\Lambda(\chi u)\in\mathcal{Q}^{\sigma,\sigma/(2b);\varphi}$ in the $\sigma\notin E$ case due to the definition of
$\mathcal{Q}^{\sigma,\sigma/(2b);\varphi}$.

Let us consider the case where $\sigma\in E$. We choose a function $\chi_1\in C^{\infty}(\overline\Omega)$ such that $\chi_1=0$ in a neighbourhood of $\Gamma$ and that $\chi_1=1$ in a neighbourhood of $\mathrm{supp}\,\chi$. The linear mapping
$M_{\chi_1}:(f,g,h)\mapsto\chi_1 (f,g,h)$ sets bounded operators
\begin{equation*}
M_{\chi_1}:\mathcal{H}^{\sigma\mp\varepsilon,(\sigma\mp\varepsilon)/(2b);\varphi}
\to\mathcal{Q}^{\sigma\mp\varepsilon,(\sigma\mp\varepsilon)/(2b);\varphi}
\end{equation*}
whenever $0<\varepsilon<1/2$ (then $\sigma\mp\varepsilon\notin E$, and $(f,g,h)$ satisfies compatibility conditions \eqref{25f8} corresponding to $s:=\sigma\mp\varepsilon$). Applying the quadratic interpolation with number parameter $1/2$ to these operators, we infer by Propositions \ref{16lem7.5} and \ref{8prop2} that the operator $M_{\chi_1}$ is also bounded on the pair of spaces
\begin{equation}\label{16f65}
\begin{aligned}
M_{\chi_1}:\mathcal{H}^{\sigma,\sigma/(2b);\varphi}&=
\bigl[
\mathcal{H}^{\sigma-\varepsilon,(\sigma-\varepsilon)/(2b);\varphi},
\mathcal{H}^{\sigma+\varepsilon,(\sigma+\varepsilon)/(2b);\varphi}
\bigr]_{1/2}\\
&\to\bigl[
\mathcal{Q}^{\sigma-\varepsilon,(\sigma-\varepsilon)/(2b);\varphi},
\mathcal{Q}^{\sigma+\varepsilon,(\sigma+\varepsilon)/(2b);\varphi}
\bigr]_{1/2}=\mathcal{Q}^{\sigma,\sigma/(2b);\varphi}.
\end{aligned}
\end{equation}
The last equality is due to the definition given by \eqref{25f10} for $s=\sigma\in E$.
Since $\chi_1=1$ in a neighbourhood of $\mathrm{supp}\,\chi$, we conclude by \eqref{16f65} that
$\Lambda(\chi u)\in \mathcal{H}^{\sigma,\sigma/(2b);\varphi}$ implies the inclusion
$\Lambda(\chi u)=\chi_1\Lambda(\chi u)\in \mathcal{Q}^{\sigma,\sigma/(2b);\varphi}$.
The required implication \eqref{16f61} is also proved in the $\sigma\in E$ case.

Now, according to formulas \eqref{16f60} and \eqref{16f61} and owing to Theorem~\ref{25th4.2}, we obtain the implication
\begin{align*}
&u\in\bigoplus_{k=1}^{N}
H^{s-\lambda+2b\varkappa_k,(s-\lambda+2b\varkappa_k)/(2b);\varphi}_{\mathrm{loc}}(\Omega_0,\Omega')\\
&\Longrightarrow
\chi u\in\bigoplus_{k=1}^{N}
H^{s-\lambda+1+2b\varkappa_k,(s-\lambda+1+2b\varkappa_k)/(2b);\varphi}(\Omega).
\end{align*}
Theorem~\ref{25th4.2} is applicable here because
\begin{equation*}
\chi u\in\bigoplus_{k=1}^{N}
H^{\sigma_0+2b\varkappa_k,(\sigma_0+2b\varkappa_k)/(2b)}(\Omega)
\end{equation*}
due to the hypothesis and because $s-\lambda+1>\sigma_{0}$. Since $\chi$ is any function of class  $C^\infty(\overline\Omega)$ and subject to $\mbox{supp}\,\chi\subset\Omega_0\cup\Omega'$, the last implication means required property \eqref{25f54}.

This property allows us to prove the theorem. Consider first the case where $s\notin\mathbb{Z}$. In this case, there exists an integer $\lambda_{0}\geq1$ such that
\begin{equation}\label{16f66}
s-\lambda_{0}<\sigma_{0}<s-\lambda_{0}+1.
\end{equation}
Using \eqref{25f54} successively for $\lambda:=\lambda_{0}$, $\lambda:=\lambda_{0}-1$,..., and $\lambda:=1$, we obtain the following chain of implications:
\begin{align*}
&u\in\bigoplus_{k=1}^{N}
H^{\sigma_0+2b\varkappa_k,(\sigma_0+2b\varkappa_k)/(2b)}(\Omega)\subset
\bigoplus_{k=1}^{N}
H^{s-\lambda_0+2b\varkappa_k,(s-\lambda_0+2b\varkappa_k)/(2b);\varphi}_{\mathrm{loc}}(\Omega_0,\Omega')
\\
&\Longrightarrow u\in\bigoplus_{k=1}^{N}
H^{s-\lambda_0+1+2b\varkappa_k,(s-\lambda_0+1+2b\varkappa_k)/(2b);\varphi}_
{\mathrm{loc}}(\Omega_0,\Omega')
\Longrightarrow\ldots\\
&\Longrightarrow
u\in\bigoplus_{k=1}^{N}
H^{s+2b\varkappa_k,(s+2b\varkappa_k)/(2b);\varphi}_{\mathrm{loc}}(\Omega_0,\Omega').
\end{align*}
Since $u$ satisfies \eqref{25f12a} by the hypothesis, the theorem is proved in this case.

Let us now consider the opposite case, when $s\in\mathbb{Z}$. No integer $\lambda_{0}$ satisfies \eqref{16f66} in this case. However, since $s-\varepsilon\notin\mathbb{Z}$ and $s-\varepsilon>\sigma_{0}$ whenever  $0<\varepsilon\ll1$, the inclusion
$$
u\in\bigoplus_{k=1}^{N}
H^{s-\varepsilon+2b\varkappa_k,(s-\varepsilon+2b\varkappa_k)/(2b);\varphi}_
{\mathrm{loc}}(\Omega_0,\Omega')
$$
holds true by what we have just proved. Hence, using \eqref{25f54} for $\lambda:=1$, we conclude that
\begin{align*}
&u\in\bigoplus_{k=1}^{N}
H^{s-\varepsilon+2b\varkappa_k,(s-\varepsilon+2b\varkappa_k)/(2b);\varphi}_
{\mathrm{loc}}(\Omega_0,\Omega')\subset
\bigoplus_{k=1}^{N}
H^{s-1+2b\varkappa_k,(s-1+2b\varkappa_k)/(2b);\varphi}_{\mathrm{loc}}(\Omega_0,\Omega')
\\
&\Longrightarrow
u\in\bigoplus_{k=1}^{N}
H^{s+2b\varkappa_k,(s+2b\varkappa_k)/(2b);\varphi}_{\mathrm{loc}}(\Omega_0,\Omega').
\end{align*}
The theorem is proved in the opposite case.
\end{proof}

\begin{proof}[Proof of Theorem~$\ref{25th4.4}$.]
Suppose that $p+2b\varkappa_k\geq0$ for certain $k\in\{1,...,N\}$. Let us first consider the case where the hypothesis $(f,g,h)\in\mathcal{Q}^{s,s/(2b);\varphi}$ is satisfied, with $s:=p+b+n/2$. By Theorem~\ref{25th4.2}, we have $u_k\in H^{\sigma,\sigma/(2b);\varphi}(\Omega)$ for $\sigma:=p+b+n/2+2b\varkappa_k$, with $\varphi$ obeying \eqref{9f4.7}. Let $w\in H^{\sigma,\sigma/(2b);\varphi}(\mathbb{R}^{n+1})$ be an extension of $u_k$. Proposition~\ref{prop-embedding}(i) states that the distribution $D_{x}^{\alpha}\partial_{t}^{\beta}w(x,t)$ is continuous on $\mathbb{R}^{n+1}$ whenever $0\leq|\alpha|+2b\beta\leq p+2b\varkappa_k$. Hence, its restriction
$D_{x}^{\alpha}\partial_{t}^{\beta}u_k(x,t)$ is continuous on $\overline\Omega$ for the same $\alpha$ and $\beta$, q.e.d.

Let us now consider the case where hypotheses \eqref{25f13}--\eqref{25f15} are satisfied. Owing to Theorem~\ref{25th4.3}, the inclusion $u_k\in H^{\sigma,\sigma/(2b);\varphi}_{\mathrm{loc}}(\Omega_0,\Omega')$ holds true for above-mentioned  $\sigma$ and $\varphi$. Let functions $\chi$ and $\eta$ be as in the proof of Theorem~\ref{25th4.3}. Then $\eta u_k\in H^{\sigma,\sigma/(2b);\varphi}(\Omega)$. Let $w\in H^{\sigma,\sigma/(2b);\varphi}(\mathbb{R}^{n+1})$ be an extension of $\eta u_k$. Proposition~\ref{prop-embedding}(i) asserts that the distribution $D_{x}^{\alpha}\partial_{t}^{\beta}w(x,t)$ is continuous on $\mathbb{R}^{n+1}$ whenever $0\leq|\alpha|+2b\beta\leq p+2b\varkappa_k$. Hence, its restriction
$D_{x}^{\alpha}\partial_{t}^{\beta}(\eta u_k)$ is continuous on $\overline\Omega$ for the same $\alpha$ and $\beta$. Thus, so is the distribution $\chi D_{x}^{\alpha}\partial_{t}^{\beta}u_k=\chi D_{x}^{\alpha}\partial_{t}^{\beta}(\eta u_k)$. The theorem is proved in view of Remark~\ref{rem-continuous-interpretation} and the arbitrariness of
the function $\chi\in C^\infty(\overline\Omega)$ subject to    $\mbox{supp}\,\chi\subset\Omega_0\cup\Omega'$.
\end{proof}

It remains to substantiate Remark~\ref{rem-sharp-condition}. Let $p\in\mathbb{Z}$, $s:=p+b+n/2>\sigma_0$ and $\varphi\in\mathcal{M}$, and assume that $p+2b\varkappa_k\geq0$ for certain $k\in\{1,...,N\}$. Suppose that implication \eqref{sharp-condition-implication-global} holds true for every function \eqref{25f12a}. We should prove that $\varphi$ satisfies \eqref{9f4.7}. To this end, we arbitrarily choose a function $w\in H^{s+2b\varkappa_k,(s+2b\varkappa_k)/(2b);\varphi}(\mathbb{R}^{n+1})$ such that $\mathrm{supp}\,w\subset\Omega$, let $u_{k}:=w\!\upharpoonright\!\Omega$ and $u:=(\delta_{1,k}\,u_{k},\ldots,\delta_{n,k}\,u_{k})$, and put $(f,g,h):=\Lambda u\in\mathcal{Q}^{s,s/(2b);\varphi}$. (As usual, $\delta_{j,k}$ is the Kronecker delta.) The vector-valued function $u$ satisfies inclusion $\eqref{25f12a}$ and the premise of implication \eqref{sharp-condition-implication-global}. Hence, $u$ satisfies the conclusion of Theorem~\ref{25th4.4}. Specifically, all generalized derivatives $(\partial^{j}u_{k})/\partial{x_{1}^{j}}$, with $0\leq j\leq p+2b\varkappa_k$, are continuous on $\overline{\Omega}$. Then the same  derivatives of $w$ are continuous on $\mathbb{R}^{n+1}$. Therefore, $\varphi$ satisfies \eqref{9f4.7} by Proposition~\ref{prop-embedding}(ii) (considered for $q:=p+2b\varkappa_k\geq0$, $\sigma:=s+2b\varkappa_k=q+b+n/2$, and $V:=\Omega$) and due to our choice of $w$. If we suppose that implication \eqref{sharp-condition-implication-local} (instead of \eqref{sharp-condition-implication-global}) holds true for every function \eqref{25f12a}, we just need to replace the condition $\mathrm{supp}\,w\subset\Omega$ with $\mathrm{supp}\,w\subset\Omega_{0}$ in this reasoning to substantiate Remark~\ref{rem-sharp-condition}.

\subsection*{Funding}

The second named author thanks the Isaac Newton Institute for Mathematical Sciences, Cambridge, for support and hospitality during the 'Solidarity Program' where the work on this paper was undertaken. This work was supported by 'EPSRC grant no. EP/R014604/1'. The author also thanks the Department of Mathematics, King’s College London, for their hospitality.
The third named author was funded by the National Academy of Sciences of Ukraine and supported by a grant from the Simons Foundation (SFI-PD-Ukraine-00014586, A.A.M.).

\subsection*{Data availability}

The authors confirm that there are no data associated with this paper.

\subsection*{Conflict of interest}

The authors have no relevant financial or non-financial interests to disclose.

\end{document}